\documentclass[11pt]{article}
\usepackage[letterpaper,margin=1in]{geometry}
\usepackage[T1]{fontenc}
\usepackage[utf8]{inputenc}
\usepackage{lmodern}
\usepackage{amsmath,amssymb,amsthm,mathtools,bm}
\usepackage{microtype}
\usepackage{needspace}
\usepackage{booktabs,array,enumitem}
\usepackage{algorithm}
\usepackage{algpseudocode}
\usepackage[numbers,sort&compress]{natbib}
\usepackage[hidelinks]{hyperref}
\allowdisplaybreaks
\setlist[itemize]{leftmargin=1.5em,itemsep=3pt,topsep=4pt}
\setlist[enumerate]{leftmargin=1.7em,itemsep=3pt,topsep=4pt}

\newtheorem{theorem}{Theorem}
\newtheorem{corollary}{Corollary}
\newtheorem{lemma}{Lemma}
\newtheorem{proposition}{Proposition}
\newtheorem{assumption}{Assumption}
\theoremstyle{definition}
\newtheorem{definition}{Definition}
\theoremstyle{remark}

\newcommand{\R}{\mathbb R}
\newcommand{\E}{\mathbb E}
\newcommand{\Prob}{\mathbb P}
\newcommand{\Y}{\mathcal Y}
\newcommand{\F}{\mathcal F}
\newcommand{\Ocal}{\mathcal O}
\newcommand{\norm}[1]{\left\|#1\right\|}
\newcommand{\ip}[2]{\left\langle#1,#2\right\rangle}
\newcommand{\bphi}{\bar\varphi}
\newcommand{\eps}{\varepsilon}
\newcommand{\hx}{\hat x}
\DeclareMathOperator{\supp}{supp}

\title{\Large\bfseries Tight Stochastic Condition-Number Dependence in\\
Nonconvex-Strongly-Concave Minimax Optimization}
\author{Qihao Zhou\thanks{Independent Researcher. Email: \texttt{qihao\_zhou@berkeley.edu}.}}
\date{}
\hypersetup{pdftitle={Tight Stochastic Condition-Number Dependence in Nonconvex-Strongly-Concave Minimax Optimization},pdfauthor={Qihao Zhou}}

\begin{document}
\maketitle

\begin{abstract}
We study whether the linear condition-number dependence in the stochastic complexity of SAPD+ is necessary for nonconvex-strongly-concave minimax optimization. For jointly $L$-smooth objectives with dual strong-concavity parameter $\mu$, we prove a lower bound that matches the SAPD+ upper bound under the same Moreau-envelope stationarity criterion and the same primal-dual initialization gap. Specifically, when $\sigma\ge\eps$, the worst-case complexity of zero-respecting algorithms is $\Theta(\kappa LG\sigma^2\eps^{-4})$ in the stated accuracy regime, where $\kappa=L/\mu$, $G$ bounds the initial primal-dual gap, and $\sigma^2$ bounds the variance of a general unbiased first-order oracle. The lower bound is realized on a smooth problem class with a bounded dual box. Our construction routes each link of a nonconvex zero-chain through a dual gradient of magnitude proportional to $\eps/\sqrt\kappa$, while an undiscovered primal coordinate prevents stationarity. It also yields the primal-gradient lower bound $\Omega(L\Delta(\sqrt\kappa\eps^{-2}+\kappa\sigma^2\eps^{-4}))$ after combination with the known deterministic bound, where $\Delta$ bounds the initial primal function gap.
\end{abstract}

\section{Introduction}\label{sec:intro}

We consider the minimax problem
\begin{equation}\label{eq:problem}
\min_{x\in\R^{d_x}}\max_{y\in\Y}f(x,y),
\end{equation}
where $f$ is jointly $L$-smooth and $\mu$-strongly concave in $y$, but may be nonconvex in $x$. Strong concavity makes the maximizer $y^*(x)$ unique and the primal function $\Phi(x):=\max_{y\in\Y}f(x,y)$ differentiable. We study stochastic first-order methods and ask how their oracle complexity depends on the condition number $\kappa=L/\mu$.

The deterministic condition-number dependence is understood up to logarithmic factors. Under their respective domain, initialization, and output guarantees, accelerated methods attain a $\widetilde O(\sqrt\kappa)$ dependence on $\kappa$ \citep{lin2020near,zhang2021complexity}. Deterministic lower bounds contain the term $\Omega(\sqrt\kappa L\Delta\eps^{-2})$, where $\Delta$ bounds the initial primal gap \citep{li2021complexity,zhang2021complexity}. The stochastic setting leaves a larger gap. For zero-respecting algorithms, the lower bound of \citet{li2021complexity} contains a stochastic term proportional to $\kappa^{1/3}$, whereas SAPD+ has a stochastic upper bound proportional to $\kappa$ \citep{zhang2022sapd}. The two results use different stationarity measures and initialization budgets, which we distinguish below. This raises the question:
\begin{center}
\emph{Is the linear dependence on $\kappa$ in the stochastic complexity of SAPD+ necessary?}
\end{center}

We answer this question affirmatively for zero-respecting algorithms under a general unbiased bounded-variance oracle. We use the criterion and initialization parameter in the SAPD+ guarantee: the Moreau envelope $\Phi_\eta$ with $\eta=1/(2L)$ and the initial primal-dual gap
\[
G_0=\max_{y\in\Y}f(0,y)-\inf_x f(x,0)
\]
Over the problem class defined in Section~\ref{sec:prelim}, with $G_0\le G$ and $\sigma\ge\eps$, we establish
\begin{equation}\label{eq:intro-tight}
\Theta\!\left(\frac{\kappa LG\sigma^2}{\eps^4}\right)
\quad\text{calls to ensure}\quad
\E\norm{\nabla\Phi_\eta(\hx)}\le\eps.
\end{equation}
The upper bound is attained by SAPD+ with a uniformly selected outer iterate. Our contribution is the matching lower bound. Thus the linear condition-number dependence persists even when the stationarity measure and initialization budget are matched exactly.

\paragraph{Why the stochastic term is linear in $\kappa$.}
A probability zero-chain conceals the next coordinate until a Bernoulli trial succeeds. The allowable failure probability depends on the size of the concealed gradient. We construct a primal function with $\Theta(L\Delta/\eps^2)$ sequential links and a gradient larger than $2\eps$ whenever its final coordinate remains undiscovered. Each link is represented by maximization over two dual variables. A separable correction removes direct communication between neighboring primal coordinates, so progress must pass through a dual gradient of size $\Theta(\eps/\sqrt\kappa)$. A variance budget $\sigma^2$ then permits an expected waiting time of $\Theta(1+\kappa\sigma^2/\eps^2)$ per link. Multiplying the chain length by this waiting time yields the desired stochastic term.

\paragraph{Contributions.}
Our main results are the following.
\begin{itemize}
\item \textbf{A matching lower bound for SAPD+.} We prove~\eqref{eq:intro-tight} for a class with bounded dual boxes and a bounded initial primal-dual gap. The lower and upper bounds use the same oracle model and Moreau-envelope criterion; the comparison does not lose a logarithmic factor.
\item \textbf{A linear-$\kappa$ primal-gradient lower bound.} Under an initial primal-gap budget $\Delta$, we prove $\Omega(L\Delta\eps^{-2}+\kappa L\Delta\sigma^2\eps^{-4})$ for $\E\norm{\nabla\Phi(\hx)}\le\eps$. Combining this with the known deterministic lower bound replaces the first term by $\Omega(\sqrt\kappa L\Delta\eps^{-2})$.
\item \textbf{A strongly concave dual probability zero-chain.} The construction combines the dual Hessian $-\mu I$ with a separable primal correction, so each link must be discovered through a dual signal of size $\Theta(\eps/\sqrt\kappa)$. Its controlled cross derivatives allow the lower bound to pass to SAPD+'s stationarity criterion with only an absolute change in accuracy. Section~\ref{sec:related} relates this construction to existing uses of stochastic concealment in dual and lower-level variables.
\end{itemize}

Section~\ref{sec:prelim} specifies the model, Section~\ref{sec:related} reviews related work, and Section~\ref{sec:results} states the results. Section~\ref{sec:construction} proves the primal-gradient lower bound. Section~\ref{sec:sapd-proof} then matches the gap and stationarity measure to complete the SAPD+ comparison. Auxiliary estimates and a matching procedure for the explicit construction are deferred to the appendix.

\section{Preliminaries}\label{sec:prelim}

We use $\norm{\cdot}$ for the Euclidean norm of a vector and the spectral norm of a matrix. For an integer $d\ge1$, let $[d]=\{1,\ldots,d\}$. For $v\in\R^d$, write $\supp(v)=\{i\in[d]:v_i\ne0\}$, and let $e_i$ denote a standard basis vector. The notation $[x;y]$ denotes the concatenation of two vectors. We use $O(\cdot)$, $\Omega(\cdot)$, and $\Theta(\cdot)$ to hide absolute constants; $\widetilde O(\cdot)$ additionally hides logarithmic factors.

\subsection{Smoothness and primal-gap class}

We impose the following assumptions on problem~\eqref{eq:problem}.

\begin{assumption}\label{asm:domain}
The set $\Y\subseteq\R^{d_y}$ is nonempty, closed, and convex, and contains the origin.
\end{assumption}

\begin{assumption}\label{asm:smooth}
The function $f$ is continuously differentiable on a neighborhood of $\R^{d_x}\times\Y$ and is jointly $L$-smooth on this domain:
\[
\norm{\nabla f(x,y)-\nabla f(x',y')}\le L\norm{[x-x';y-y']}
\]
for all $x,x'\in\R^{d_x}$ and $y,y'\in\Y$.
\end{assumption}

\begin{assumption}\label{asm:concave}
For every $x\in\R^{d_x}$, the function $f(x,\cdot)$ is $\mu$-strongly concave on $\Y$, where $0<\mu\le L$:
\[
f(x,y)\le f(x,y')+\ip{\nabla_yf(x,y')}{y-y'}-\frac\mu2\norm{y-y'}^2
\]
for all $y,y'\in\Y$.
\end{assumption}

\begin{assumption}\label{asm:gap}
The primal function $\Phi(x)=\max_{y\in\Y}f(x,y)$ is bounded below and satisfies
\[
\Phi(0)-\inf_{x\in\R^{d_x}}\Phi(x)\le\Delta.
\]
\end{assumption}

Under Assumptions~\ref{asm:domain}--\ref{asm:concave}, the maximizer $y^*(x)$ exists and is unique. Indeed, since $0\in\Y$, strong concavity gives
\[
f(x,y)\le f(x,0)+\ip{\nabla_y f(x,0)}{y}-\frac\mu2\norm{y}^2.
\]
The right-hand side tends to $-\infty$ as $\norm{y}\to\infty$, so continuity and closedness of $\Y$ give existence; strong concavity gives uniqueness. This estimate also bounds the maximizers locally uniformly in $x$. Danskin's theorem therefore gives $\nabla\Phi(x)=\nabla_x f(x,y^*(x))$, including when $\Y$ is unbounded; see \citet[Lemma~4.3]{lin2020gradient} for the bounded-domain case. Let $\F(L,\mu,\Delta)$ denote the class of problems satisfying Assumptions~\ref{asm:domain}--\ref{asm:gap}, with arbitrary finite dimensions and admissible dual sets. We write $\kappa=L/\mu$.

\subsection{Oracle model and algorithms}

\begin{definition}[Stochastic first-order oracle]\label{def:oracle-model}
An oracle $\Ocal$ for $f$ has variance at most $\sigma^2$ if, for every feasible query $(x,y)$,
\begin{equation}\label{eq:oracle-model}
\E_\xi[\Ocal(x,y;\xi)]=\nabla f(x,y),\qquad
\E_\xi\norm{\Ocal(x,y;\xi)-\nabla f(x,y)}^2\le\sigma^2.
\end{equation}
Each call uses a fresh independent sample $\xi$.
\end{definition}

The model in Definition~\ref{def:oracle-model} permits arbitrary unbiased gradient estimators. It does not require $\Ocal(x,y;\xi)$ to be the gradient of a differentiable sample loss, and it does not impose mean-squared smoothness. The lower bound below is proved for this oracle model.

\begin{definition}[Zero-respecting algorithm]\label{def:zero}
An algorithm starts at $(x^0,y^0)=(0,0)$ and has zero-query output $\hx^0=0$. Every query $(x^t,y^t)$ lies in $\R^{d_x}\times\Y$, and every output $\hx^t$ lies in $\R^{d_x}$. Write $(g_x^s,g_y^s)=\Ocal(x^s,y^s;\xi^s)$. The algorithm is zero-respecting if, for every $t\ge1$, its next query and its output after $t$ calls satisfy
\begin{align}\label{eq:zero}
\supp(x^t)\cup\supp(\hx^t)&\subseteq\bigcup_{s=0}^{t-1}\supp(g_x^s),&
\supp(y^t)&\subseteq\bigcup_{s=0}^{t-1}\supp(g_y^s).
\end{align}
Queries may be adaptive and may use internal randomness independent of the oracle samples.
\end{definition}

This definition includes gradient methods and many momentum and coordinatewise adaptive methods. Projection onto a box containing the origin also preserves the property. The definition allows a discovered coordinate to be reset to zero; discovery refers to the union of past oracle supports, rather than the support of the current query.

We measure complexity by a deterministic upper bound on the number of oracle calls needed to meet the stated expected-norm stationarity criterion. An algorithm that stops earlier can be padded with queries at the origin. A lower bound is dimension-free: the dimensions of a hard instance may depend on the problem parameters and on the target accuracy.

\subsection{The class and criterion for the SAPD+ comparison}

To compare with SAPD+, we use the initial primal-dual gap
\begin{equation}\label{eq:pdgap}
G_0:=\max_{y\in\Y}f(0,y)-\inf_{x\in\R^{d_x}}f(x,0).
\end{equation}
Let $\mathcal M(L,\mu,G)$ consist of problems satisfying Assumptions~\ref{asm:domain}--\ref{asm:concave} and the following conditions:
\begin{enumerate}[label=(\roman*)]
\item $\Y$ is a bounded box containing the origin;
\item $f$ is twice continuously differentiable on an open neighborhood of $\R^{d_x}\times\Y$;
\item $f(\cdot,y)$ is bounded below for every $y\in\Y$, and $G_0\le G$.
\end{enumerate}
These conditions ensure applicability of the SAPD+ result used below. Dimensions and box radii are unrestricted. Since $\Phi(x)\ge f(x,0)$, every member of $\mathcal M(L,\mu,G)$ has primal gap at most $G$. We retain separate symbols for the two budgets: a bound on the primal gap alone need not bound $G_0$.

The primal function is $L$-weakly convex: $\Phi(\cdot)+(L/2)\norm{\cdot}^2$ is a pointwise maximum of convex functions. Set $\eta=1/(2L)$ and define
\begin{equation}\label{eq:envelope}
\Phi_\eta(x)=\min_z\left\{\Phi(z)+\frac{\norm{z-x}^2}{2\eta}\right\}.
\end{equation}
The minimizer $z=\operatorname{prox}_{\eta\Phi}(x)$ is unique, and $\nabla\Phi_\eta(x)=(x-z)/\eta$. The SAPD+ comparison uses
\begin{equation}\label{eq:env-criterion}
\E\norm{\nabla\Phi_\eta(\hx)}\le\eps,
\end{equation}
whereas the primal-gradient lower bound uses
\begin{equation}\label{eq:criterion}
\E\norm{\nabla\Phi(\hx)}\le\eps.
\end{equation}
Write $\mathfrak N_{\rm env}(\mathcal M,\sigma,\eps)$ for the smallest deterministic call budget of a zero-respecting algorithm that guarantees~\eqref{eq:env-criterion} for every problem in $\mathcal M$ and every oracle satisfying Definition~\ref{def:oracle-model}. The bound is uniform over dimensions; the algorithm may know the class parameters and the dual box.

\section{Related Work}\label{sec:related}

\paragraph{Algorithms for nonconvex-strongly-concave minimax optimization.}
\citet{lin2020gradient} analyzed gradient descent ascent for nonconvex-concave and nonconvex-strongly-concave objectives. Accelerated proximal-point methods improve the deterministic condition-number dependence to $\widetilde O(\sqrt\kappa)$ \citep{lin2020near,zhang2021complexity}. The cited guarantees retain their own initialization and output conditions: \citet[Theorem~20]{lin2020near} state a constant-success-probability guarantee, while \citet[Corollary~4.1]{zhang2021complexity} retain an initial dual-error term. In the stochastic setting, SAPD+ combines an inexact proximal-point framework with stochastic accelerated primal-dual updates \citep{zhang2022sapd}. For all SAPD+ theorem numbers and output guarantees in this paper, we use the October 13, 2024 revision, arXiv:2205.15084v4, of that work. Its Theorems~1 and~3 and Remark~3 give the Moreau-envelope guarantee for a uniformly selected outer iterate. We use this guarantee directly and match its primal-dual gap parameter. Under additional oracle regularity, variance reduction gives different accuracy dependence \citep{luo2020stochastic}; our oracle model does not impose that regularity.

\paragraph{Lower bounds.}
Smooth nonconvex zero-chains provide stationary-point lower bounds for deterministic first-order methods \citep{carmon2020lower}. Randomly concealing the next chain coordinate produces the stochastic $\eps^{-4}$ lower bound under bounded variance \citep{arjevani2023lower}. For nonconvex-strongly-concave minimax optimization, deterministic lower bounds have the factor $\sqrt\kappa$ \citep{li2021complexity,zhang2021complexity}, while the stochastic term of \citet{li2021complexity} has the factor $\kappa^{1/3}$. Definition~4 of the latter work explicitly restricts its first-order algorithms to be zero-respecting. Its stochastic Theorem~2 uses a projected-gradient mapping on a constrained primal domain. Our Theorem~\ref{thm:lower} gives linear dependence on $\kappa$ for zero-respecting algorithms on an unconstrained primal domain, measured by $\norm{\nabla\Phi}$. The comparison concerns the stochastic condition-number dependence within this algorithmic framework.

\paragraph{Recent nonconvex-concave lower bounds.}
\citet{wu2026lower} establish deterministic and stochastic lower bounds for nonconvex-concave minimax problems with a bounded dual domain, using Moreau-envelope stationarity. Their stochastic construction preserves the primal value function, clips quadratic interactions along dual paths to bound the randomized gradients, and conceals the next dual coordinate with a Bernoulli oracle. These features are closely related to the mechanism used here. Their bounds are parameterized by the dual diameter and do not impose uniform strong concavity. Our construction uses two dual variables per primal link and the diagonal dual Hessian $-\mu I$. A separable correction preserves the desired primal chain while making the concealed signal proportional to $\eps/\sqrt\kappa$. The resulting bounds control joint smoothness, cross derivatives, and the initial primal-dual gap simultaneously. This strongly concave realization gives the linear-$\kappa$ lower bound and the comparison with SAPD+; a lower bound over the larger nonconvex-concave class alone does not imply a lower bound over its strongly concave subclass.

\paragraph{Bilevel and nonconvex-Polyak--\L{}ojasiewicz problems.}
\citet{ji2026lower} establishes deterministic and stochastic first-order lower bounds for nonconvex-strongly-convex bilevel optimization. The revised work of \citet{chen2026condition}, arXiv:2511.22331v4, merges the earlier Chen--Zhang preprint with that ICML paper and strengthens their results. Its stochastic first-order construction also conceals lower-level gradient information. These bilevel constructions use distinct upper- and lower-level objectives; minimax optimization requires the additional identity that the lower-level objective is the negative of the upper-level objective. Consequently, their bilevel lower bounds do not directly establish the minimax bound proved here. Separately, \citet{pan2026lower} prove a deterministic lower bound for nonconvex-Polyak--\L{}ojasiewicz minimax optimization. That result concerns a weaker dual regularity condition and does not establish the stochastic strongly concave bound considered here.

\Needspace{7\baselineskip}
\section{Main Results}\label{sec:results}

Our main result establishes the worst-case optimality of SAPD+ among zero-respecting algorithms over $\mathcal M(L,\mu,G)$ in the stochastic regime. Throughout the paper, $C=370$ is a fixed numerical constant.

\begin{theorem}[Tight stochastic complexity and SAPD+]\label{thm:sapd-tight}
Let $L,\mu,G,\eps>0$ and $\sigma\ge0$ satisfy
\begin{equation}\label{eq:tight-regime}
\kappa=\frac L\mu\ge C,\qquad
\eps^2\le\frac{LG}{440C},\qquad \sigma\ge\eps.
\end{equation}
For $\eta=1/(2L)$, the worst-case stochastic first-order oracle complexity over $\mathcal M(L,\mu,G)$ is
\begin{equation}\label{eq:main-tight}
\mathfrak N_{\rm env}(\mathcal M(L,\mu,G),\sigma,\eps)
=\Theta\!\left(\frac{\kappa LG\sigma^2}{\eps^4}\right).
\end{equation}
The upper bound is attained by SAPD+ with a uniformly selected outer iterate. The lower bound holds even when the primal component of the oracle is deterministic.
\end{theorem}

The upper bound in Theorem~\ref{thm:sapd-tight} is a specialization of Theorems~1 and~3 in the arXiv v4 revision of \citet{zhang2022sapd}. The new lower bound shows that the stochastic dependence on $\kappa$ in this guarantee cannot be improved for zero-respecting algorithms. It holds on a subclass with $\norm{\nabla^2_{xy}f}\le\sqrt{\mu L}$, so the difficulty does not require the largest cross derivatives allowed by joint $L$-smoothness. Section~\ref{sec:sapd-proof} gives the complete comparison, including the gap and stationarity conversion.

\begin{table}[htbp]
\centering
\caption{Matching bounds for the SAPD+ comparison under~\eqref{eq:tight-regime}. Both rows use $\mathcal M(L,\mu,G)$, variance at most $\sigma^2$, and $\E\norm{\nabla\Phi_{1/(2L)}(\hx)}\le\eps$.}\label{tab:results}
\small
\begin{tabular}{@{}lll@{}}
\toprule
Result & Oracle complexity & Algorithm scope\\
\midrule
SAPD+ \citep{zhang2022sapd} & $O(\kappa LG\sigma^2\eps^{-4})$ & A zero-respecting method\\[3pt]
Our lower bound & $\Omega(\kappa LG\sigma^2\eps^{-4})$ & All zero-respecting methods\\
\bottomrule
\end{tabular}
\end{table}

The construction also gives a lower bound under the usual primal-gap budget, for the gradient of $\Phi$ itself. This quantitative result is the main ingredient in the proof of Theorem~\ref{thm:sapd-tight}.

\begin{theorem}[Primal-gradient lower bound]\label{thm:lower}
Let $C=370$. For any $L,\mu,\Delta,\eps>0$ and $\sigma\ge0$ satisfying
\begin{equation}\label{eq:regime}
\kappa=\frac L\mu\ge C,\qquad
\eps^2\le\frac{L\Delta}{96C},
\end{equation}
there exist a problem $f\in\F(L,\mu,\Delta)$ with a bounded dual box and an oracle of variance at most $\sigma^2$ such that every zero-respecting algorithm satisfies $\E\norm{\nabla\Phi(\hx^N)}>\eps$ whenever
\begin{equation}\label{eq:lower-budget}
N\le c_0\left(\frac{L\Delta}{\eps^2}
+\frac{\kappa L\Delta\sigma^2}{\eps^4}\right),
\qquad c_0=\frac{1}{1536e^2C^2}.
\end{equation}
The primal component of the oracle is deterministic.
\end{theorem}

For zero-respecting algorithms, the stochastic term in~\eqref{eq:lower-budget} has linear dependence on $\kappa$, improving the $\kappa^{1/3}$ dependence in the stochastic lower bound of \citet{li2021complexity}, with the domain and stationarity distinctions described in Section~\ref{sec:related}. No lower bound on the noise level is required: when $\sigma=0$, our oracle is deterministic. The constants in~\eqref{eq:regime} are chosen for a direct verification of joint smoothness rather than optimized numerical values.

\begin{corollary}[Combined worst-case lower bound]\label{cor:combined}
Under~\eqref{eq:regime}, the worst-case oracle complexity over $\F(L,\mu,\Delta)$, for zero-respecting algorithms satisfying~\eqref{eq:criterion}, is at least
\begin{equation}\label{eq:combined}
\Omega\!\left(L\Delta\left[\frac{\sqrt\kappa}{\eps^2}
+\frac{\kappa\sigma^2}{\eps^4}\right]\right).
\end{equation}
\end{corollary}

\begin{proof}
The deterministic lower bound of \citet[Theorem~1]{li2021complexity} gives the first term. Its parameter assumptions are $L,\mu,\Delta,\eps>0$ and $\kappa\ge1$, which follow from~\eqref{eq:regime}. Its hard instance has domain $\R^{d_x}\times\R^{d_y}$ and its conclusion is stated directly in terms of $\norm{\nabla\Phi}$. In particular, projection onto the primal domain is the identity, so the projected-gradient criterion introduces no conversion factor. The zero-chain obstruction is pointwise for every admissible support pattern and therefore also applies to internally randomized zero-respecting methods and their outputs. Its exact-gradient oracle has zero variance and is admissible for every $\sigma\ge0$. Theorem~\ref{thm:lower} gives the second term. Both constructions belong to the class considered here, though they may use different dual domains. The worst-case complexity is at least the maximum of these two lower bounds, and $\max\{a,b\}\ge(a+b)/2$ for $a,b\ge0$.
\end{proof}

The stochastic term in~\eqref{eq:combined} dominates when $\sigma^2$ is at least of order $\eps^2/\sqrt\kappa$. In particular, when $\sigma\ge\eps$, the lower bound reduces to $\Omega(\kappa L\Delta\sigma^2\eps^{-4})$.

Theorem~\ref{thm:sapd-tight} concerns the regime $\sigma\ge\eps$. Theorem~\ref{thm:lower} covers every $\sigma\ge0$, but does not assert that SAPD+ has optimal deterministic complexity. Appendix~\ref{app:gradient-tight} gives a matching primal-gradient result under controlled cross derivatives. A construction-specific upper bound is recorded separately in Appendix~\ref{sec:matching}.

\section{The Dual Probability Zero-Chain}\label{sec:construction}

We prove Theorem~\ref{thm:lower} in four steps. First, choose a nonconvex chain for which an undiscovered last coordinate prevents stationarity. Second, express its links through maximization over dual variables while keeping the primal part separable. Third, conceal only the next dual link with a Bernoulli oracle. Finally, count the successful discoveries. The scales are
\begin{equation}\label{eq:roadmap}
\underbrace{T=\Theta(L\Delta/\eps^2)}_{\text{links to discover}},\qquad
\underbrace{c^2=\Theta(\eps^2/\kappa)}_{\text{dual signal squared}},\qquad
\underbrace{p^{-1}=\Theta(1+\sigma^2/c^2)}_{\text{calls per link}}.
\end{equation}
Their product is the lower-bound scale. The next subsections justify each step; scalar derivative estimates are deferred to the appendix.

\subsection{Step 1: A chain that certifies nonstationarity}

Following \citet{carmon2020lower,arjevani2023lower}, define
\begin{equation}\label{eq:scalar}
\psi(t)=\begin{cases}
0,&t\le1/2,\\
\exp\big(1-(2t-1)^{-2}\big),&t>1/2,
\end{cases}
\qquad
\varphi(t)=\sqrt e\int_{-\infty}^t e^{-s^2/2}\,\mathrm ds.
\end{equation}
Let $\varphi_0=\varphi(0)=\sqrt{\pi e/2}$ and $\bphi(t)=\varphi(t)-\varphi_0$. The function $\bphi$ is odd and vanishes at zero. For $u\in\R^T$, set
\begin{equation}\label{eq:chain}
F_T(u)=-\varphi(u_1)+\sum_{i=1}^{T-1}
\big[\psi(-u_i)\varphi(-u_{i+1})-\psi(u_i)\varphi(u_{i+1})\big].
\end{equation}
Here $\psi(1)=1$, so the initial term agrees with the usual definition of this chain.

\Needspace{10\baselineskip}
\begin{lemma}[Basic properties]\label{lem:basic}
The functions $\psi$ and $\varphi$ are infinitely differentiable. They satisfy
\begin{align*}
&\psi(t)=\psi'(t)=0\quad(t\le1/2),\qquad
0\le\psi(t)<e,\quad 0\le\psi'(t)\le\sqrt{54/e},\quad |\psi''(t)|\le32.5,\\
&0<\varphi(t)<\sqrt{2\pi e},\qquad
0<\varphi'(t)\le\sqrt e,\qquad |\varphi''(t)|\le1.
\end{align*}
Moreover, $F_T(0)-\inf_uF_T(u)\le12T$, and
\begin{equation}\label{eq:chain-gradient}
|u_T|<1\quad\Longrightarrow\quad\norm{\nabla F_T(u)}>1.
\end{equation}
\end{lemma}

For completeness, Appendix~\ref{app:basic} proves the properties needed here. Condition~\eqref{eq:chain-gradient} ensures that an algorithm must discover the final coordinate before it can produce a sufficiently small primal gradient.

\subsection{Step 2: Realizing each link through the dual variables}

Given the parameters in Theorem~\ref{thm:lower}, define
\begin{equation}\label{eq:parameters}
L_0=\frac LC,\quad \lambda=\frac{2\eps}{L_0},\quad
c=\lambda\sqrt{\mu L_0},\quad R=\frac{5c}{\mu},\quad
T=\left\lfloor\frac{L_0\Delta}{48\eps^2}\right\rfloor.
\end{equation}
The conditions in~\eqref{eq:regime} imply $\mu\le L_0$ and $T\ge2$. We take $x\in\R^T$ and $y=[y^+;y^-]\in\Y=[-R,R]^{2(T-1)}$. Each pair $(y_i^+,y_i^-)$ connects primal coordinates $i$ and $i+1$.

For $u=x/\lambda$, define the vector $q(u)=[q^+(u);q^-(u)]$ by
\begin{equation}\label{eq:q}
q_i^+(u)=\psi(u_i)-\bphi(u_{i+1}),\qquad
q_i^-(u)=\psi(-u_i)+\bphi(-u_{i+1}),\qquad i\in[T-1].
\end{equation}
Let
\begin{align}\label{eq:H}
H(u)&=-\varphi(u_1)-\sum_{i=1}^{T-1}S_i(u),\\
S_i(u)&=\frac12\psi(u_i)^2+\frac12\psi(-u_i)^2
+\varphi_0\big(\psi(u_i)-\psi(-u_i)\big)+\bphi(u_{i+1})^2.\notag
\end{align}
The objective is
\begin{equation}\label{eq:f}
f(x,y)=L_0\lambda^2H(x/\lambda)+c\ip{q(x/\lambda)}{y}
-\frac\mu2\norm y^2.
\end{equation}
The important feature of $H$ is that it is separable in all primal coordinates. Its gradient cannot propagate information from one primal coordinate to the next.

\begin{lemma}[Primal function]\label{lem:primal}
For every $x\in\R^T$, the maximizer is $y^*(x)=(c/\mu)q(x/\lambda)$ and lies in the interior of $\Y$. Furthermore,
\begin{equation}\label{eq:primal}
\Phi(x)=L_0\lambda^2F_T(x/\lambda).
\end{equation}
\end{lemma}

The identity behind Lemma~\ref{lem:primal} is
\begin{equation}\label{eq:cancellation}
H(u)+\frac12\norm{q(u)}^2=F_T(u).
\end{equation}
Since $c^2/\mu=L_0\lambda^2$, maximizing over the dual variables adds precisely the square term in~\eqref{eq:cancellation}. The expansion recovers the links in $F_T$, while the separable terms cancel. Equivalently,
\begin{equation}\label{eq:completed-square}
f(x,y)=L_0\lambda^2F_T(x/\lambda)
-\frac\mu2\norm{y-\frac c\mu q(x/\lambda)}^2.
\end{equation}

\begin{lemma}[Membership in the problem class]\label{lem:class}
The function in~\eqref{eq:f}, with the parameters in~\eqref{eq:parameters}, belongs to $\F(L,\mu,\Delta)$. It also belongs to $\mathcal M(L,\mu,1.14\Delta)$ and satisfies
\begin{equation}\label{eq:instance-cross}
\norm{\nabla^2_{xy}f(x,y)}\le8.8\sqrt{\mu L_0}<\sqrt{\mu L}
\qquad\text{for all }(x,y)\in\R^T\times\Y.
\end{equation}
\end{lemma}

The proofs of Lemmas~\ref{lem:primal} and~\ref{lem:class} are given in Appendix~\ref{app:objective}. Strong concavity is immediate from $\nabla^2_{yy}f=-\mu I$. Joint smoothness requires more care because the primal Hessian contains terms proportional to $cy_i^\pm\psi''(x_i/\lambda)/\lambda^2$. The box constraint and the choice of $c$ give $cR/\lambda^2=5L_0$, so these terms are uniformly bounded. The resulting estimate is $\norm{\nabla^2f}<370L_0=L$ throughout $\R^T\times\Y$.

\subsection{Step 3: Concealing the next dual link}

Define the largest nonzero primal coordinate by
\[
j(x)=\max\{i\in[T]:x_i\ne0\},\qquad j(0)=0.
\]
For $u=x/\lambda$, let
\begin{equation}\label{eq:v}
v(x)=\begin{cases}
c\big(\psi(u_j)e_{y_j^+}+\psi(-u_j)e_{y_j^-}\big),&1\le j=j(x)\le T-1,\\
0,&j(x)\in\{0,T\}.
\end{cases}
\end{equation}
Here $e_{y_j^\pm}$ denotes the corresponding basis vector in the dual space. Choose
\begin{equation}\label{eq:p}
p=\frac{e^2c^2}{e^2c^2+\sigma^2}
=\left(1+\frac{\kappa\sigma^2}{4e^2C\eps^2}\right)^{-1},
\end{equation}
and define
\begin{equation}\label{eq:oracle}
\Ocal(x,y;\xi)=
\begin{bmatrix}
\nabla_x f(x,y)\\
\nabla_y f(x,y)+(\xi/p-1)v(x)
\end{bmatrix},\qquad \xi\sim\operatorname{Bernoulli}(p).
\end{equation}
Thus only part of the dual gradient at the current frontier is randomized. The primal gradient is returned exactly.

\begin{lemma}[Oracle validity]\label{lem:oracle}
The oracle in~\eqref{eq:oracle} satisfies Definition~\ref{def:oracle-model} for every $\sigma\ge0$.
\end{lemma}

\begin{proof}
The identity $\E(\xi/p-1)=0$ gives unbiasedness. At most one of $\psi(u_j)$ and $\psi(-u_j)$ is nonzero, so $\norm{v(x)}\le ec$, including the endpoint cases in~\eqref{eq:v}. Therefore,
\[
\E\norm{\Ocal(x,y;\xi)-\nabla f(x,y)}^2
=\frac{1-p}{p}\norm{v(x)}^2
=\frac{\sigma^2}{e^2c^2}\norm{v(x)}^2\le\sigma^2.
\]
When $\sigma=0$, we have $p=1$ and the oracle is deterministic.
\end{proof}

The scale of the randomized component is
\[
c=2\eps\sqrt{\frac{C}{\kappa}}.
\]
By contrast, Lemmas~\ref{lem:basic} and~\ref{lem:primal} give $\norm{\nabla\Phi(x)}>2\eps$ whenever $x_T=0$. This separation between the dual information and the primal stationarity measure produces the factor $\kappa$ in the stochastic lower bound.

\subsection{Step 4: Counting discoveries}

Let
\[
I_x^t=\bigcup_{s=0}^{t-1}\supp(g_x^s),\qquad
I_y^t=\bigcup_{s=0}^{t-1}\supp(g_y^s),\qquad
P_t=\max I_x^t,
\]
with $P_t=0$ when $I_x^t$ is empty. A dual link $k$ is discovered at time $t$ if either of its two coordinates belongs to $I_y^t$.

\begin{lemma}[Dual probability zero-chain]\label{lem:progress}
For every zero-respecting algorithm interacting with~\eqref{eq:oracle},
\begin{equation}\label{eq:progress}
P_N\le1+\sum_{t=0}^{N-1}\xi^t\qquad\text{for every }N\ge1.
\end{equation}
\end{lemma}

The dependency order is
\[
x_1\ \longrightarrow\ y_1^\pm\ \longrightarrow\ x_2\
\longrightarrow\ \cdots\ \longrightarrow\ y_{T-1}^\pm\
\longrightarrow\ x_T.
\]
The first primal coordinate is revealed at the origin. For subsequent calls, the key invariant is that every discovered dual link has index at most $P_t$. At a frontier $P=P_t<T$, there are two cases:
\begin{enumerate}[label=(\roman*)]
\item If dual link $P$ is undiscovered, then its coordinates must be zero in the query. No new primal coordinate can appear, and link $P$ can be revealed only when $\xi^t=1$.
\item If dual link $P$ is already discovered, then the next primal coordinate $P+1$ may appear. However, dual link $P+1$ cannot appear on the same call, because both $x_{P+1}^t$ and all later primal coordinates are still zero in that query.
\end{enumerate}
Consequently, each new primal coordinate after the first requires a distinct earlier Bernoulli success. This gives~\eqref{eq:progress}. Appendix~\ref{app:progress} verifies the invariant from the gradient formulas, including for adaptive queries that reset previously discovered coordinates to zero.

\subsection{Proof of Theorem~\ref{thm:lower}}

\begin{proof}
Lemmas~\ref{lem:class} and~\ref{lem:oracle} verify the objective and oracle assumptions. At zero queries, $\hx^0=0$ and $\norm{\nabla\Phi(0)}=2\sqrt e\,\eps>\eps$.

Now suppose $1\le N\le(T-1)/(2p)$. By Lemma~\ref{lem:progress} and Markov's inequality,
\[
\Prob(P_N\ge T)
\le\Prob\left(\sum_{t=0}^{N-1}\xi^t\ge T-1\right)
\le\frac{Np}{T-1}\le\frac12.
\]
The zero-respecting property implies $\hx_T^N=0$ on the event $P_N<T$. On this event, Lemmas~\ref{lem:basic} and~\ref{lem:primal} imply
\[
\norm{\nabla\Phi(\hx^N)}
=L_0\lambda\norm{\nabla F_T(\hx^N/\lambda)}>2\eps.
\]
Consequently, $\E\norm{\nabla\Phi(\hx^N)}>\eps$.

Finally, $a=L_0\Delta/(48\eps^2)\ge2$ gives $T=\lfloor a\rfloor\ge a/2$ and $T-1\ge T/2$. Hence $T-1\ge L_0\Delta/(192\eps^2)$. Substituting~\eqref{eq:p} and $L_0=L/C$ yields
\[
\frac{T-1}{2p}\ge
c_0\left(\frac{L\Delta}{\eps^2}
+\frac{\kappa L\Delta\sigma^2}{\eps^4}\right),
\]
with $c_0=1/(1536e^2C^2)$, which proves~\eqref{eq:lower-budget}.
\end{proof}

\section{Matching the SAPD+ Guarantee}\label{sec:sapd-proof}

Theorem~\ref{thm:lower} uses a primal-gap budget and the gradient of $\Phi$. To prove Theorem~\ref{thm:sapd-tight}, we must transfer its hard instance to the primal-dual gap budget and the Moreau-envelope criterion. Lemma~\ref{lem:class} already gives $G_0\le1.14\Delta$. The remaining step is a comparison of the two gradients with a constant independent of $\kappa$.

\subsection{Comparing stationarity at the same point}

\begin{lemma}[Gradient comparison]\label{lem:comparison}
Suppose $f\in\mathcal M(L,\mu,G)$ and $\norm{\nabla^2_{xy}f}\le b$ throughout its domain. Then $\Phi$ is $(L+b^2/\mu)$-smooth. For $\eta=1/(2L)$,
\begin{equation}\label{eq:comparison-general}
\norm{\nabla\Phi(x)}\le\left(1+\eta\left(L+\frac{b^2}{\mu}\right)\right)
\norm{\nabla\Phi_\eta(x)}\qquad\text{for every }x.
\end{equation}
In particular, if $b\le\sqrt{\mu L}$, then
\begin{equation}\label{eq:comparison}
\norm{\nabla\Phi(x)}\le2\norm{\nabla\Phi_\eta(x)}.
\end{equation}
\end{lemma}

\begin{proof}
Let $y=y^*(x)$ and $y'=y^*(x')$. The optimality inequalities for the two constrained maximizations and $\mu$-strong concavity imply
\[
\mu\norm{y-y'}^2
\le\ip{\nabla_y f(x,y')-\nabla_y f(x',y')}{y-y'}
\le b\norm{x-x'}\norm{y-y'}.
\]
Thus $y^*$ is $(b/\mu)$-Lipschitz, including at the boundary of the box. Danskin's formula gives
\begin{align*}
\norm{\nabla\Phi(x)-\nabla\Phi(x')}
&\le L\norm{x-x'}+b\norm{y^*(x)-y^*(x')}\\
&\le(L+b^2/\mu)\norm{x-x'}.
\end{align*}
For $z=\operatorname{prox}_{\eta\Phi}(x)$, the first-order optimality condition is
\[
\nabla\Phi(z)=\frac{x-z}{\eta}=\nabla\Phi_\eta(x).
\]
Combining this identity with the smoothness just proved yields
\[
\norm{\nabla\Phi(x)}
\le\norm{\nabla\Phi(z)}+(L+b^2/\mu)\norm{x-z}
=\big(1+\eta(L+b^2/\mu)\big)\norm{\nabla\Phi_\eta(x)}.
\]
If $b^2\le\mu L$, the factor is at most $2$.
\end{proof}

For a general jointly smooth problem, the same calculation can lose a factor of order $\kappa$. Our instance instead satisfies~\eqref{eq:instance-cross}, so~\eqref{eq:comparison} transfers its lower bound with an absolute factor. The comparison holds at the algorithm's output; no proximal computation is added to the oracle model.

\subsection{Proof of Theorem~\ref{thm:sapd-tight}}

\begin{proof}
\textbf{Lower bound.} Apply the construction of Theorem~\ref{thm:lower} with
\[
\Delta'=G/1.14,\qquad \eps_{\rm lb}=2\eps
\]
in place of $\Delta$ and $\eps$. The parameter restriction is valid because
\[
4\cdot96\cdot1.14=437.76<440,\qquad
\eps_{\rm lb}^2\le\frac{L\Delta'}{96C}.
\]
By Lemma~\ref{lem:class}, the constructed problem lies in $\mathcal M(L,\mu,G)$ and satisfies~\eqref{eq:comparison}. Hence, for any zero-respecting algorithm, Theorem~\ref{thm:lower} gives
\[
N\le\frac{c_0}{16\cdot1.14}\frac{\kappa LG\sigma^2}{\eps^4}
\quad\Longrightarrow\quad
\E\norm{\nabla\Phi_\eta(\hx^N)}
\ge\frac12\E\norm{\nabla\Phi(\hx^N)}>\eps.
\]
This proves the lower bound for exactly the criterion in~\eqref{eq:env-criterion}.

\textbf{Upper bound.} The SAPD+ guarantee of \citet[Theorems~1 and~3]{zhang2022sapd}, specialized to $\gamma=L$, gives~\eqref{eq:env-criterion} in
\begin{equation}\label{eq:sapd-upper}
O\!\left(\kappa\max\{1,\sigma^2/\eps^2\}
\left(1+\frac{LG}{\eps^2}\right)\right)
\end{equation}
calls. Appendix~\ref{app:sapd-specialization} verifies the assumptions, the uniform-iterate output, and the zero-respecting implementation. Under~\eqref{eq:tight-regime}, $\sigma^2/\eps^2\ge1$ and $LG/\eps^2\ge440C$, so~\eqref{eq:sapd-upper} is $O(\kappa LG\sigma^2\eps^{-4})$. Combining the two bounds proves~\eqref{eq:main-tight}.
\end{proof}

\section{Conclusion}\label{sec:conclusion}

We establish a lower bound that matches the stochastic complexity of SAPD+ for zero-respecting algorithms, using the same Moreau-envelope criterion and primal-dual initialization gap. The key construction places the information needed for primal progress in a dual gradient of magnitude $\Theta(\eps/\sqrt\kappa)$. This yields a linear-$\kappa$ stochastic lower bound, even with an exact primal oracle. The same instance gives a primal-gradient lower bound under a primal-gap budget and, under controlled cross derivatives, a matching primal-gradient result.

The established tightness concerns general unbiased bounded-variance estimators and a class allowing bounded dual boxes. Extending the lower bound to arbitrary randomized algorithms, stochastic gradients of differentiable sample losses, or an unconstrained dual domain remains open for this construction. Likewise, a matching primal-gradient upper bound under only a primal-gap budget requires a separate argument.

\clearpage
\section*{AI Assistance Acknowledgement}

The author used GPT-6 Astra and Claude Science with Opus 5.5 to assist with drafting this paper. All content, including all proofs, was verified by the author, who takes full responsibility for the paper.

\begingroup
\small
\setlength{\bibsep}{3pt plus 1pt minus 1pt}
\bibliographystyle{plainnat}
\bibliography{references}
\endgroup

\clearpage
\appendix
\section{Properties of the Nonconvex Chain}\label{app:basic}

\begin{proof}[Proof of Lemma~\ref{lem:basic}]
The function $\exp(-s^{-2})$, extended by zero for $s\le0$, is infinitely differentiable and all of its derivatives vanish at zero. This proves smoothness of $\psi$ at $1/2$. For $t>1/2$, writing $s=2t-1$ gives
\[
\psi'(t)=\frac{4\psi(t)}{s^3},\qquad
\psi''(t)=\frac{8\psi(t)(2-3s^2)}{s^6}.
\]
The maximum of $4e\exp(-s^{-2})s^{-3}$ is $\sqrt{54/e}$. For the second derivative, let $r=s^{-2}$. The stationary points of $(2r^3-3r^2)e^{-r}$ on $(0,\infty)$ occur at
\[
r_-=(9-\sqrt{33})/4,\qquad r_+=(9+\sqrt{33})/4.
\]
Evaluation at these two points, together with the limits at zero and infinity, gives
\[
|\psi''(t)|\le
8e\max_{r\in\{r_-,r_+\}}e^{-r}|2r^3-3r^2|<32.5.
\]
The bounds on $\psi$ and its vanishing derivatives on $(-\infty,1/2]$ follow directly from~\eqref{eq:scalar}.

For the second scalar function,
\[
\varphi'(t)=\sqrt e\,e^{-t^2/2},\qquad
\varphi''(t)=-t\sqrt e\,e^{-t^2/2}.
\]
The stated derivative bounds follow by maximizing these expressions. The Gaussian integral gives $0<\varphi(t)<\sqrt{2\pi e}$, and symmetry gives $\varphi(-t)=2\varphi_0-\varphi(t)$. Thus $\bphi$ is odd and $|\bphi(t)|<\varphi_0$.

For the function gap, each negative link term in~\eqref{eq:chain} has magnitude less than $e\sqrt{2\pi e}<12$, and every positive link term is nonnegative. Therefore $F_T(u)>-12T$. Since $F_T(0)=-\varphi_0<0$, we obtain $F_T(0)-\inf_uF_T(u)\le12T$.

It remains to prove~\eqref{eq:chain-gradient}. All contributions to every partial derivative of $F_T$ are nonpositive. Indeed, for $2\le i\le T$, the incoming contribution is
\[
-\psi(-u_{i-1})\varphi'(-u_i)-\psi(u_{i-1})\varphi'(u_i),
\]
and, for $i<T$, the outgoing contribution is
\[
-\psi'(-u_i)\varphi(-u_{i+1})-\psi'(u_i)\varphi(u_{i+1}).
\]
At $i=1$, the initial contribution is $-\varphi'(u_1)$. If $|u_T|<1$, let $i$ be the first index with $|u_i|<1$. If $i=1$, then $\partial_1F_T(u)\le-\varphi'(u_1)<-1$. Otherwise $|u_{i-1}|\ge1$, so
\[
\psi(u_{i-1})+\psi(-u_{i-1})\ge\psi(1)=1.
\]
Since $\varphi'$ is even and $\varphi'(u_i)>1$, the incoming contribution to $\partial_iF_T(u)$ is less than $-1$. The other contributions are nonpositive. Hence $\norm{\nabla F_T(u)}>1$ in either case.
\end{proof}

\section{Verification of the Objective}\label{app:objective}

\subsection{Primal identity}

\begin{proof}[Proof of Lemma~\ref{lem:primal}]
For a fixed $x$, the unconstrained dual maximizer of~\eqref{eq:f} is
\[
y_i^{+*}(x)=\frac c\mu\big(\psi(u_i)-\bphi(u_{i+1})\big),\qquad
y_i^{-*}(x)=\frac c\mu\big(\psi(-u_i)+\bphi(-u_{i+1})\big).
\]
Lemma~\ref{lem:basic} implies
\[
|y_i^{\pm*}(x)|<\frac c\mu(e+\varphi_0)<\frac{4.8c}{\mu}<R.
\]
Thus this point is also the constrained maximizer and lies in the interior of the box.

To verify~\eqref{eq:cancellation}, write $a=\psi(u_i)$, $b=\psi(-u_i)$, and $v=u_{i+1}$. Expanding the two squares gives
\begin{align*}
\frac12(a-\bphi(v))^2+\frac12(b+\bphi(-v))^2
&=\frac12a^2+\frac12b^2+\bphi(v)^2-(a+b)\bphi(v)\\
&=b\varphi(-v)-a\varphi(v)+S_i(u).
\end{align*}
Summing over $i$ and using~\eqref{eq:H} proves the identity. Since $c^2=\mu L_0\lambda^2$, substitution of the maximizer into~\eqref{eq:f} yields
\[
\Phi(x)=L_0\lambda^2\left(H(u)+\frac12\norm{q(u)}^2\right)
=L_0\lambda^2F_T(u).
\]
\end{proof}

\subsection{Joint smoothness, strong concavity, and the gap}

\begin{proof}[Proof of Lemma~\ref{lem:class}]
The dual set is a bounded box containing zero, so Assumption~\ref{asm:domain} holds. The function is infinitely differentiable on the ambient space, and $\nabla_{yy}^2f=-\mu I$ proves Assumption~\ref{asm:concave}.

To verify joint smoothness on the prescribed domain, write
\[
\nabla^2 f(x,y)=\begin{bmatrix}A&B\\B^\top&-\mu I\end{bmatrix},
\qquad A=\nabla^2_{xx}f(x,y),\quad B=\nabla^2_{xy}f(x,y).
\]
Every summand in~\eqref{eq:f} depends on at most one primal coordinate, so $A$ is diagonal. Let $h_k(u)=\partial_k H(u)$ and use the convention $y_0^\pm=y_T^\pm=0$. Then
\begin{align*}
A_{kk}=L_0\partial_k h_k(u)+\frac c{\lambda^2}
\big[&y_k^+\psi''(u_k)+y_k^-\psi''(-u_k)\\
&-y_{k-1}^+\varphi''(u_k)+y_{k-1}^-\varphi''(-u_k)\big].
\end{align*}
At most one of the two signs of $\psi$, $\psi'$, and $\psi''$ can contribute at a given $u_k$. Differentiating~\eqref{eq:H} and applying Lemma~\ref{lem:basic} gives
\[
|\partial_k h_k(u)|
\le1+\frac{54}{e}+32.5e+32.5\varphi_0+2(e+\varphi_0)<187.
\]
The remaining terms are bounded using $|y_i^\pm|\le R=5c/\mu$:
\[
\frac c{\lambda^2}
\left|y_k^+\psi''(u_k)+y_k^-\psi''(-u_k)
-y_{k-1}^+\varphi''(u_k)+y_{k-1}^-\varphi''(-u_k)\right|
\le\frac{5c^2}{\mu\lambda^2}(32.5+2)=172.5L_0.
\]
Consequently, $\norm A\le359.5L_0$.

For each link $i\in[T-1]$, the nonzero cross derivatives are
\begin{align*}
B_{i,y_i^+}&=\frac c\lambda\psi'(u_i),&
B_{i,y_i^-}&=-\frac c\lambda\psi'(-u_i),\\
B_{i+1,y_i^+}&=-\frac c\lambda\varphi'(u_{i+1}),&
B_{i+1,y_i^-}&=-\frac c\lambda\varphi'(-u_{i+1}).
\end{align*}
The absolute row sums are at most $12.3c/\lambda$, and the absolute column sums are at most $6.2c/\lambda$. For the row bound we use the deliberately loose estimate $2(4.46)+2\sqrt e<12.3$; using the fact that at most one of $\psi'(u_i)$ and $\psi'(-u_i)$ is nonzero would improve this constant, but is unnecessary for the stated smoothness bound. Hence
\[
\norm B\le\sqrt{12.3\cdot6.2}\,\frac c\lambda
\le8.8\sqrt{\mu L_0}\le8.8L_0.
\]
By splitting the Hessian into its diagonal blocks and off-diagonal blocks,
\[
\norm{\nabla^2 f(x,y)}
\le\max\{\norm A,\mu\}+\norm B
\le368.3L_0<370L_0=L.
\]
The domain $\R^T\times\Y$ is convex, so integration along line segments proves Assumption~\ref{asm:smooth}.

Finally, Lemmas~\ref{lem:basic} and~\ref{lem:primal} imply
\[
\Phi(0)-\inf_x\Phi(x)
\le12L_0\lambda^2T
=\frac{48\eps^2T}{L_0}\le\Delta.
\]
The primal function is bounded below because $F_T$ is bounded below. This proves Assumption~\ref{asm:gap}.

For the additional claims, $\norm B\le8.8\sqrt{\mu L_0}<\sqrt{\mu L}$ proves~\eqref{eq:instance-cross}. All scalar components of $H$ and $q$ are bounded, so $f(\cdot,y)$ is bounded below for every fixed $y\in\Y$. Finally, since at most one of $\psi(u_i)$ and $\psi(-u_i)$ is nonzero,
\[
S_i(u)\le\frac{e^2}{2}+e\varphi_0+\varphi_0^2<13.6.
\]
Using $\Phi(0)=-L_0\lambda^2\varphi_0$ and the definition of $H$ gives
\begin{align*}
G_0&\le L_0\lambda^2\big(\sqrt{2\pi e}-\varphi_0+13.6(T-1)\big)\\
&\le13.6L_0\lambda^2T
\le\frac{13.6}{12}\Delta<1.14\Delta.
\end{align*}
The domain and differentiability requirements are already verified. Hence $f\in\mathcal M(L,\mu,1.14\Delta)$.
\end{proof}

\section{Proof of the Chain Invariant}\label{app:progress}

\begin{proof}[Proof of Lemma~\ref{lem:progress}]
With $y_0^\pm=y_T^\pm=0$, differentiation of~\eqref{eq:f} gives
\begin{align}\label{eq:gx}
\partial_{x_k}f(x,y)=L_0\lambda h_k(u)+\frac c\lambda
\big[&y_k^+\psi'(u_k)-y_k^-\psi'(-u_k)\notag\\
&-y_{k-1}^+\varphi'(u_k)-y_{k-1}^-\varphi'(-u_k)\big]
\end{align}
for $k\in[T]$, and
\begin{align}\label{eq:gy}
\partial_{y_k^+}f(x,y)&=c\big(\psi(u_k)-\bphi(u_{k+1})\big)-\mu y_k^+,\\
\partial_{y_k^-}f(x,y)&=c\big(\psi(-u_k)+\bphi(-u_{k+1})\big)-\mu y_k^-.\notag
\end{align}
for $k\in[T-1]$. Separability of $H$ implies that $h_k$ depends only on $u_k$. In addition, $h_k(u)=0$ when $u_k=0$ and $k\ge2$.

The first query is at the origin. Equations~\eqref{eq:gx}--\eqref{eq:gy} give $g_x^0=-L_0\lambda\sqrt e\,e_1$ and $g_y^0=0$. Thus $P_1=1$ and no dual link is discovered.

We prove the following invariant until the first time $P_t=T$: every discovered dual link has index at most $P_t$. Fix a time $t$ for which $P=P_t<T$ and the invariant holds. All primal indices below lie in $[T]$, and all dual-link indices lie in $[T-1]$. The support restrictions give $x_k^t=0$ and $y_k^{t,\pm}=0$ for indices $k>P$.

If link $P$ has not been discovered, then $y_P^{t,\pm}=0$ as well. Equation~\eqref{eq:gx} shows that $g_{x_k}^t=0$ for every $k\ge P+1$, so $P_{t+1}=P$. Equation~\eqref{eq:gy} shows that no dual link with index larger than $P$ can be revealed. At link $P$, its exact dual gradient is $c\psi(\pm u_P)$. If $x_P^t=0$, both components vanish. Otherwise $j(x^t)=P$, and the oracle returns the two components as $(\xi^t/p)c\psi(\pm u_P)$. Hence link $P$ can first be discovered only when $\xi^t=1$. The invariant is preserved.

If link $P$ has already been discovered, then~\eqref{eq:gx} gives $g_{x_k}^t=0$ for all $k\ge P+2$, and therefore $P_{t+1}\le P+1$. For each $k\ge P+1$, the entries $u_k,u_{k+1},y_k^{t,\pm}$ vanish, so~\eqref{eq:gy} gives zero on link $k$. The perturbation in~\eqref{eq:oracle} also vanishes there because $j(x^t)\le P$. Thus no dual link with index larger than $P$ is discovered on this call, and the invariant again holds.

Every increase from $P$ to $P+1$ therefore requires a distinct earlier success. Indeed, link $P$ cannot be discovered before primal coordinate $P$, or on the same call that reveals that coordinate. Its subsequent discovery requires $\xi^t=1$ while $P_t=P$. Since $P_t$ is nondecreasing, counting these discoveries gives $P_N-1\le\sum_{t=0}^{N-1}\xi^t$ until $P_N=T$. Thereafter $P_N=T$ and at least $T-1$ successes have already occurred, so the inequality continues to hold. This proves~\eqref{eq:progress}.
\end{proof}

\section{Specialization of the SAPD+ Upper Bound}\label{app:sapd-specialization}

We give the parameter correspondence used in~\eqref{eq:sapd-upper}. All references to the SAPD+ paper in this appendix use arXiv:2205.15084v4 of \citet{zhang2022sapd}. In their composite formulation, take the primal regularizer to be zero and the dual regularizer to be the indicator $\iota_{\Y}$. Their subdifferentials contain zero at every point of their respective domains, so the bounded-minimum-subgradient condition in their Assumption~5 holds. Their Theorem~3 then permits the unbounded primal domain. The definition of $\mathcal M$ supplies the required lower boundedness of every primal slice.

Joint smoothness permits all four block Lipschitz bounds to be set to $L$, and the weak-convexity parameter can be set to $\gamma=L$. Their equation~(4) adds $(\mu_x+\gamma)\norm{x-\bar x}^2/2$ to the objective, and their envelope parameter is $(\mu_x+\gamma)^{-1}$. Thus $\mu_x=\gamma=L$ gives the proximal term $L\norm{x-\bar x}^2$ and exactly $\eta=1/(2L)$. Their gap function, defined after equation~(6), specializes at $(x_0,y_0)=(0,0)$ to
\[
\mathcal G(0,0)=\sup_{y\in\Y}f(0,y)-\inf_{x\in\R^{d_x}}f(x,0)=G_0\le G.
\]

For a target $\delta>0$, write $\tau_x,\tau_y$ for the two step sizes to avoid confusing a step size with our noise budget $\sigma$. Set both partial-gradient variance bounds to $\sigma^2$, which is valid under Definition~\ref{def:oracle-model}. Their Theorem~1 explicitly sets the momentum parameter to $\theta=1$. Specializing its equation~(7) and the outer-iteration bound following equation~(8) gives
\begin{align*}
\tau_x&=\min\left\{\frac1{4L},\frac{\delta^2}{480L\sigma^2}\right\},&
\tau_y&=\min\left\{\frac1{3L},\frac{\delta^2}{4512L\sigma^2}\right\},\\
N_{\rm in}&=\left\lceil33\max\left\{\frac4{L\tau_x},\frac8{\mu\tau_y}\right\}\right\rceil,&
K&=\left\lceil\frac{96LG}{\delta^2}+1\right\rceil,
\end{align*}
The constants $480$, $4512$, and $33$ come from their equation~(7), and $96$ comes from the bound following their equation~(8). When $\sigma=0$, the fractions involving $1/\sigma^2$ are interpreted as $+\infty$. The resulting iteration counts satisfy
\[
N_{\rm in}=O\!\left(\kappa\max\{1,\sigma^2/\delta^2\}\right),\qquad
K+1=O(1+LG/\delta^2).
\]
Theorems~1 and~3 of the cited work ensure that the outer iterates $x^0,\ldots,x^K$ obey
\[
\frac1{K+1}\sum_{t=0}^K\E\norm{\nabla\Phi_{1/(2L)}(x^t)}^2\le\delta^2.
\]
For an independent uniform index $J\in\{0,\ldots,K\}$, Jensen's inequality gives $\E\norm{\nabla\Phi_{1/(2L)}(x^J)}\le\delta$. No gradient-norm evaluation is needed to select the output.

Each partial-gradient evaluation can be implemented with a fresh full-oracle call and discarding the unused component. Definition~\ref{def:oracle-model} then gives the required conditional unbiasedness and variance bounds for each partial gradient. An inner iteration uses only a constant number of such calls. The primal updates use observed gradients and previously formed iterates; the dual projection is onto a box containing zero. Both preserve zero coordinates until the corresponding oracle coordinate has appeared. Averages and the uniformly selected outer iterate have the same property. Thus this implementation is zero-respecting and uses
\[
O\!\left(\kappa\max\{1,\sigma^2/\delta^2\}(1+LG/\delta^2)\right)
\]
calls. Setting $\delta=\eps$ proves~\eqref{eq:sapd-upper}.

\section{A Matching Primal-Gradient Result}\label{app:gradient-tight}

Define the subclass
\[
\mathcal C(L,\mu,G)=\{f\in\mathcal M(L,\mu,G):
\norm{\nabla^2_{xy}f(x,y)}\le\sqrt{\mu L}\ \text{throughout the domain}\}.
\]
Lemma~\ref{lem:comparison} gives a constant comparison between primal and envelope gradients on this class.

\begin{corollary}[Primal-gradient tightness under controlled cross derivatives]\label{cor:gradient-tight}
Under~\eqref{eq:tight-regime}, the worst-case oracle complexity of zero-respecting algorithms over $\mathcal C(L,\mu,G)$ for~\eqref{eq:criterion} is
\[
\Theta\!\left(\frac{\kappa LG\sigma^2}{\eps^4}\right).
\]
The upper bound is attained by SAPD+ run to envelope accuracy $\eps/2$.
\end{corollary}

\begin{proof}
For the upper bound, use Appendix~\ref{app:sapd-specialization} with $\delta=\eps/2$ and apply~\eqref{eq:comparison} at its output. The call budget remains $O(\kappa LG\sigma^2\eps^{-4})$. For the lower bound, the instance used in the proof of Theorem~\ref{thm:sapd-tight} lies in $\mathcal C$ and forces $\E\norm{\nabla\Phi(\hx)}>2\eps$ for $\Omega(\kappa LG\sigma^2\eps^{-4})$ calls. It therefore also rules out~\eqref{eq:criterion} within that budget.
\end{proof}

This corollary uses a primal-dual gap budget. The distinction from a primal-gap budget is substantive: for $b>0$, the function
\[
f(x,y)=\frac L2x^2+by-\frac\mu2 y^2,\qquad
\Y=[-b/\mu,b/\mu],
\]
has $\Phi(0)-\inf_x\Phi(x)=0$ but $G_0=b^2/(2\mu)$. Hence the stated SAPD+ guarantee does not yield a bound depending only on the primal gap by replacing $G$ with $\Delta$. This example concerns the initialization parameter in that guarantee; it does not establish a lower bound on the actual number of iterations taken by SAPD+.

\begingroup
\setlength{\abovedisplayskip}{8pt plus 2pt minus 2pt}
\setlength{\belowdisplayskip}{8pt plus 2pt minus 2pt}
\setlength{\abovedisplayshortskip}{0pt plus 2pt}
\setlength{\belowdisplayshortskip}{5pt plus 2pt minus 2pt}
\section{A Matching Bound for the Construction}\label{sec:matching}

This supplementary result concerns the explicit objective and oracle used to prove Theorem~\ref{thm:lower}. Let $N_*(f,\Ocal;\eps)$ be the smallest deterministic call budget with which a zero-respecting algorithm can satisfy~\eqref{eq:criterion} on this pair. The algorithm in the upper bound is given the construction and its parameters.

\begin{proposition}[Matching bound for the construction]\label{thm:pair}
Under~\eqref{eq:regime}, let $(f,\Ocal)$ be the pair defined in Section~\ref{sec:construction}, with chain length $T$ and discovery probability $p$. Then
\begin{equation}\label{eq:pair-finite}
\frac{T-1}{2p}<N_*(f,\Ocal;\eps)
\le\left\lceil8\left(T+\frac{T-1}{p}\right)\right\rceil.
\end{equation}
Consequently, for every $\sigma\ge0$,
\begin{equation}\label{eq:pair-rate}
N_*(f,\Ocal;\eps)
=\Theta\!\left(\frac{L\Delta}{\eps^2}
+\frac{\kappa L\Delta\sigma^2}{\eps^4}\right).
\end{equation}
\end{proposition}

This upper bound uses the specified oracle and a procedure given the explicit construction. It does not assert a uniform upper bound over $\F(L,\mu,\Delta)$. The absence of $\sqrt\kappa$ in the deterministic term is compatible with Corollary~\ref{cor:combined}: this particular family is easier without noise than the deterministic worst-case instances.

We now prove Proposition~\ref{thm:pair}. The lower bound is already contained in the preceding proof. To establish the upper bound, we give a procedure that first discovers all primal coordinates and then returns an explicit point with small primal gradient. This additional step is necessary because discovering a coordinate does not itself imply stationarity.

\subsection{Discovering the coordinates}

The first query at the origin returns
\[
g_x^0=-L_0\lambda\sqrt e\,e_1,\qquad g_y^0=0,
\]
so coordinate $1$ is discovered. Suppose coordinates $1,\ldots,j$ have been discovered and $j<T$. Query
\[
x=\lambda\sum_{i=1}^j e_i,\qquad y=0.
\]
The returned dual frontier coordinate is
\[
g_{y_j^+}=\frac{\xi c}{p},
\]
because $\psi(1)=1$ and $x_{j+1}=0$. Repeating this query therefore reveals $y_j^+$ after a geometric number of trials with success probability $p$. Once the link has been discovered, keep the same $x$ and query at $y=(R/2)e_{y_j^+}$. The primal oracle then returns
\begin{equation}\label{eq:reveal}
g_{x_{j+1}}=-\frac{cR}{2\lambda}\varphi'(0)\ne0,
\end{equation}
revealing the next primal coordinate. All queries satisfy the support restriction and the box constraint.

Let $\tau$ be the number of calls until all $T$ primal coordinates have been discovered by this procedure. With the geometric distribution supported on $\{1,2,\ldots\}$,
\begin{equation}\label{eq:waiting}
\tau=T+\sum_{j=1}^{T-1}G_j,\qquad
G_j\overset{\mathrm{iid}}\sim\operatorname{Geom}(p),\qquad
m:=\E\tau=T+\frac{T-1}{p}.
\end{equation}
The $T$ deterministic calls consist of the initial query and one primal-revelation query for each link.

\begin{algorithm}[t]
\caption{A zero-respecting procedure for the constructed pair}\label{alg:discovery}
\small
\begin{algorithmic}[1]
\Require Parameters $T,\lambda,R,p$ from the construction
\State $N\gets\lceil8(T+(T-1)/p)\rceil$
\State Query $\Ocal(0,0)$; set $n\gets1$
\For{$j=1,\ldots,T-1$}
  \State $x\gets\lambda\sum_{i=1}^j e_i$
  \Repeat
    \If{$n=N$} \State \Return $0$ \EndIf
    \State $(g_x,g_y)\gets\Ocal(x,0)$; $n\gets n+1$
  \Until{$g_{y_j^+}\ne0$}
  \If{$n=N$} \State \Return $0$ \EndIf
  \State Query $\Ocal(x,(R/2)e_{y_j^+})$; $n\gets n+1$
\EndFor
\State \Return $10\lambda\sqrt T\,\mathbf1_T$
\end{algorithmic}
\end{algorithm}

\subsection{Returning an approximate stationary point}

\begin{lemma}[Terminal point]\label{lem:terminal}
For $M=10\sqrt T$ and $x_{\mathrm{good}}=\lambda M\mathbf1_T$,
\begin{equation}\label{eq:terminal}
\norm{\nabla\Phi(x_{\mathrm{good}})}\le\frac{\eps}{10T}<\frac\eps2,
\qquad \norm{\nabla\Phi(0)}=2\sqrt e\,\eps.
\end{equation}
\end{lemma}

\begin{proof}
Since $M\ge10$, all negative-argument $\psi$ and $\psi'$ terms vanish. Differentiating~\eqref{eq:chain} gives, for each $i\in[T]$,
\[
|[\nabla F_T(M\mathbf1_T)]_i|
\le e\sqrt e\,e^{-M^2/2}+\sqrt{2\pi e}\,\psi'(M).
\]
We have $\psi'(M)=4\psi(M)/(2M-1)^3\le4e/M^3$ and $e^{-M^2/2}\le M^{-3}$. Using $e\sqrt e+4e\sqrt{2\pi e}<50$, we obtain
\[
\norm{\nabla F_T(M\mathbf1_T)}
\le\frac{50\sqrt T}{M^3}=\frac{1}{20T}.
\]
Multiplication by $L_0\lambda=2\eps$ proves the first claim. The second follows from $\nabla F_T(0)=-\sqrt e\,e_1$.
\end{proof}

\begin{proof}[Proof of Proposition~\ref{thm:pair}]
\enlargethispage{\baselineskip}
The lower bound in~\eqref{eq:pair-finite} follows from the proof of Theorem~\ref{thm:lower}, including the zero-query case. For the upper bound, run Algorithm~\ref{alg:discovery}, which truncates the discovery procedure after $N=\lceil8m\rceil$ calls. The output is $x_{\mathrm{good}}$ if discovery finishes and zero otherwise. The terminal point is legal because all primal coordinates have been discovered; zero is always a legal output. By~\eqref{eq:waiting} and Markov's inequality,
\[
\Prob(\tau>N)\le\frac{m}{N}\le\frac18.
\]
Lemma~\ref{lem:terminal} therefore gives
\[
\E\norm{\nabla\Phi(\hx)}
\le\frac\eps2+\frac18(2\sqrt e\,\eps)
=\left(\frac12+\frac{\sqrt e}{4}\right)\eps<\eps.
\]
This proves~\eqref{eq:pair-finite} with a deterministic call budget.

Finally, $T\ge2$ and $0<p\le1$ imply that both bounds in~\eqref{eq:pair-finite} are of order $T/p$. By~\eqref{eq:parameters} and~\eqref{eq:p},
\[
T=\Theta\!\left(\frac{L\Delta}{\eps^2}\right),\qquad
\frac1p=1+\frac{\kappa\sigma^2}{4e^2C\eps^2}.
\]
The constant $C=370$ is fixed, so their product gives~\eqref{eq:pair-rate} with absolute implicit constants.
\end{proof}
\endgroup

\end{document}